\documentclass{article}%
\usepackage{amsmath}%
\usepackage{amsfonts}%
\usepackage{amssymb}%
\usepackage{graphicx}
\providecommand{\U}[1]{\protect \rule{.1in}{.1in}}

\def\e{\mathbb{E}}
\def\d{d}
\def\x{\underline{\sigma}}
\def\s{\overline{\sigma}}
\def\E{\hat{\e}}

\newtheorem{theorem}{Theorem}[section]

\newtheorem{corollary}[theorem]{Corollary}

\newtheorem{definition}[theorem]{Definition}
\newtheorem{example}[theorem]{Example}

\newtheorem{lemma}[theorem]{Lemma}

\newtheorem{proposition}[theorem]{Proposition}
\newtheorem{remark}[theorem]{Remark}

\newenvironment{proof}[1][Proof]{\noindent \textbf{#1.} }{$\Box$}
\begin{document}

\title{On the strict comparison theorem for $G$-expectations}
\author{Xinpeng LI
\\School  of Mathematics, Shandong University\\ 250100, Jinan,
China} \date{January 20, 2010 } \maketitle

\begin{quotation}
\textbf{Abstract. } This paper investigates the strict comparison
theorem under the framework of $G$-expectation, i.e., let $X\leq Y$
q.s., if $X,Y$ satisfy some additional conditions, then
$\E[X]<\E[Y]$.
\end{quotation}

\section{Introduction}

In 2006, to deal with the model uncertainty problem in finance, Peng
establish the sublinear expectation theory and introduce a new
sublinear expectation, called $G$-expectation, which have many well
properties as classical linear expectation except linearity (see
Peng \cite{Peng2006a,Peng2007,Peng2006b,PengSurvey,Peng2009}).
Unlike the well-known $g$-expectation, $G$-expectation was
introduced via fully nonlinear parabolic partial differential
equations.

In this paper, we consider the following problem: for a linear
expectation $E_P$, we know that if $X\leq Y$ then $E_P[X]\leq
E_P[Y]$, furthermore, if $X\leq Y$ and $P(X<Y)>0$, then
$E_P[X]<E_P[Y]$. If we replace the linear expectation $E_P$ by
Peng's $G$-expectation $\E$, the former holds obviously, we interest
that when the latter holds. We give three forms of strict comparison
theorem for $G$-expectation and some interesting examples.

This paper is organized as follows: in Section 2, we recall some
basic notions and results of $G$-expectation. In Section 3, we prove
strict comparison theorems for $G$-expectation and give  some
interesting examples.

\section{Preliminaries}
We present some preliminaries in the theory of $G$-expectations.
More details can be found in Peng
\cite{Peng2006a,Peng2007,Peng2006b,PengSurvey,Peng2009}.

 Let $\Omega$ be a given set and let $\mathcal{H}$ be a linear
space of real valued functions defined on $\Omega$ satisfying: if
$X_{i}\in \mathcal{H}$, $i=1,\cdots,d$,
then%
\[
\varphi(X_{1},\cdots,X_{d})\in \mathcal{H}, {\  \forall }\varphi \in
C_{l,Lip}(\mathbb{R}^{d}),
\]
where $C_{l,Lip}(\mathbb{R}^{d})$ is the space of all real
continuous functions defined on $\mathbb{R}$ such that
$$|\varphi(x)-\varphi(y)|\leq C(1+|x|^k+|y|^k)|x-y|, \forall
x,y\in\mathbb{R}^d, \  k\  \text{depends\ on}\ \varphi.$$

{{A {sublinear expectation }$\mathbb{{E}}$ on $\mathcal{H}$ is a
functional $\mathbb{{E}}:\mathcal{H}\mapsto \mathbb{R}$ satisfying
the following properties:

\noindent{\textup{(1)} Monotonicity:} \  \ \  \  \  \  \  \  \  \ \
\ \ If $X\geq Y$ then $\mathbb{{E}}[X]\geq \mathbb{{E}}[Y].$\newline
{\textup{(2)} Constant preserving: \ \ }\ \
$\mathbb{{E}}[c]=c$.\newline {\textup{(3)}} {Sub-additivity: \ \ \ \
}}}\  \  \  \  \  \  \  \ $\mathbb{{E} }[X+Y]\leq
\mathbb{{E}}[X]+\mathbb{{E}}[Y].$\newline{{{\textup{(4)} Positive
homogeneity: } \ $\mathbb{{E}}[\lambda X]=\lambda \mathbb{{E}
}[X]$,$\  \  \forall \lambda \geq0$.}}

\begin{definition}
In a sublinear expectation space $(\Omega,\mathcal{H},\e)$, a random
variable $Y$
 is said to be independent from another
 random variable $X$ if
\[
\mathbb{{E}}[\varphi(X,Y)]=\mathbb{{E}}[\mathbb{{E}}%
[\varphi(x,Y)]_{x=X}],\  \  \ {\forall }\varphi \in
C_{l,Lip}(\mathbb{R}\times \mathbb{R}).
\]

Two random variables $X_{1}$ and $X_{2}$  are called identically
distributed, denoted by $X_{1}\sim X_{2}$, if
\[
\mathbb{{E}}[\varphi(X_{1})]=\mathbb{{E}}[\varphi (X_{2})],\ \ \
{\forall}\varphi \in C_{l,Lip}(\mathbb{R}).
\]

If $\bar{X}$ is identically distributed with $X$ and independent
from $X$, then $\bar{X}$ is said to be an independent copy of $X$.
\end{definition}

\begin{definition}
In a sublinear expectation space $(\Omega,\mathcal{H},\e)$, a random
variable $X\in\mathcal{H}$ is said to be G-normal distributed with
$\s^2=\e[X^2]$ and $\x^2=-\e[-X^2]$, denoted by
$X\sim\mathcal{N}(0,[\x^2,\s^2])$, if
$$aX+b\bar{X}\sim\sqrt{a^2+b^2}X,\ \ \ \forall a,b\geq 0,$$
where $\bar{X}$ is an independent copy of $X$.
\end{definition}

We can give a characterization of $G$-normal distribution by fully
nonlinear parabolic partial differential equation (see Peng
\cite{Peng2006a,Peng2007,Peng2006b,PengSurvey,Peng2009}).

\begin{proposition}
Let $X\sim\mathcal{N}(0,[\x^2,\s^2])$, then for each $\varphi\in
C_{l,Lip}(\mathbb{R})$, the function $u$ defined by
$$u(t,x):=\e[\varphi(x+\sqrt{t}X)], t\geq 0, x\in\mathbb{R} ,$$
is the unique viscosity solution of the following G-heat equation:
\begin{equation}
\left\{\begin{array}{cl}\partial_tu-G(\partial_{xx}u)=0, &(t,x)\in [0,\infty)\times \mathbb{R}, \label{e1}\\
u|_{t=0}=\varphi, &  \end{array}\right.\
\end{equation}
where $G(\alpha)=(\s^2\alpha^+-\x^2\alpha^-)/2$.
\end{proposition}

We will give the notion of $G$-expectations. Let
$\Omega=C_{0}(\mathbb{R}^{+})$ be the space of all
$\mathbb{R}$-valued continuous paths $(\omega_{t})_{t\in
\mathbb{R}^{+}}$ with $\omega_{0}=0$. We consider the canonic
process: $B_t(\omega)=\omega_t, t\in[0,\infty)$, $\omega\in \Omega$.
For each fixed $T>0$, we set
$$Lip(\Omega _T):=\{\varphi(B_{t_1},B_{t_2},\cdots,B_{t_n}):\forall n\geq 1, t_1,\cdots, t_n\in[0,T],\forall \varphi\in C_{l,Lip}(\mathbb{R}^n)\}.$$
and
$$Lip(\Omega):=\bigcup_{n=1}^\infty Lip(\Omega_n).$$
Then we can construct a consistent sublinear expectation called
$G$-expectation $\E[\cdot]$ on $Lip(\Omega)$, such that $B_1$ is
$G$-normal distributed under $\E$ and for each $s,t\geq 0$ and
$t_1,\cdots,t_N\in[0,t]$, we have
$$\E[\varphi(B_{t_1},\cdots,B_{t_N},B_{t+s}-B_t)]=\E[\psi(B_{t_1},\cdots, B_{t_N})],$$
where
$\psi(x_1,\cdots,x_N)=\E[\varphi(x_1,\cdots,x_N,\sqrt{s}B_1)].$
Under $G$-expectation $\E[\cdot]$, the canonic process
$\{B_t:t\geq0\}$ is called $G$-Brownian motion.

The completion of $Lip(\Omega)$ under the Banach norm $\E[|\cdot|]$
is denoted by $L_G^1(\Omega)$. $\E[\cdot]$ can be extended uniquely
to a sublinear expectation on $L_G^1(\Omega)$ (see Peng
\cite{Peng2006b,PengSurvey,Peng2009}).

We denote by $\mathcal{B}(\Omega)$ the Borel $\sigma$-algebra of
$\Omega$. It was proved in Hu and Peng \cite{HP} (see also Denis, Hu
and Peng \cite{DHP}) that there exists a weakly compact family
$\mathcal{P}$ of probability measures defined on $(\Omega
,\mathcal{B}(\Omega))$ such that
\[
\mathbb{\hat{E}}[X]=\sup_{P\in \mathcal{P}}E_{P}[X],\ \ \forall X\in
L_G^1(\Omega).
\]

For such weakly compact family $\mathcal{P}$, we can  introduce the
natural Choquet capacity
\[
v(A):=\sup_{P\in \mathcal{P}}P(A),\  \ A\in \mathcal{B}(\Omega).
\]

\begin{definition}
A set $A\subset \Omega$ is polar if $v(A)=0$. A property holds
quasi-surely (q.s.) if it holds outside a polar set.
\end{definition}

\begin{definition}
\label{def} A real function $X$ on $\Omega$ is said to be
quasi-continuous if for each $\varepsilon>0$, there exists an open
set $O$ with $v(O)<\varepsilon$ such that $X|_{O^c}$ is continuous.
\end{definition}

The following proposition can be found in Denis, Hu and Peng
\cite{DHP}.
\begin{proposition}
\label{p1} For each $X\in L^1_G(\Omega)$, there exists $Y$ such that
$Y=X$ q.s. and $Y$ is quasi-continuous.
\end{proposition}

\section{Main Theorem}
In this section, we consider strict comparison theorem for Peng's
$G$-expectation $\E$, where $\E$ defined on $(\Omega,L_G^1(\Omega))$
and there exists a weakly compact family $\mathcal{P}$ such that
$\E[X]=\sup_{P\in\mathcal{P}}E_P[X]$.

\begin{theorem}
\label{t1} Let $X,Y\in L^1_G(\Omega)$ and $X\leq Y$ q.s. If
$$\inf_{P\in\mathcal{P}}P(X<Y)>0,$$ then  $\E[X]<\E[Y]$.
\end{theorem}

\begin{proof}
Since $\E[X]-\E[Y]\leq\E[X-Y]$, we only consider the case of $X\leq
0$ q.s. and $\inf_{P\in\mathcal{P}}P(X<0)>0$.

For such $X$, we can choose $\varepsilon$ such that
$0<\varepsilon<\inf_{P\in\mathcal{P}}P(X<0)$. Since $X\in
L^1_G(\Omega)$, by Proposition \ref{p1}, there exists $Y$ such that
$X=Y$ q.s. and $Y$ is quasi-continuous. Noting that $X=Y$ q.s.
implies $\E[X]=\E[Y]$, and if $\E[Y]<0$ we also have $\E[X]<0$.
Without loss of generality we can assume that $X$ is
quasi-continuous. By Definition \ref{def}, there exists an open set
$O$ such that $v(O)<\inf_{P\in\mathcal{P}}P(X<0)-\varepsilon$ and
$X|_{O^c}$ is continuous.

Set $A=\{\omega:X(\omega)\geq 0\}$, $A_n=\{\omega:X(\omega)\geq-
{1}/{n}\}$, $F=A\cap O^c$ and $F_n=A_n\cap O^c$, then we can check
that $F_n$ is closed and $F_n\downarrow F$. Since $\mathcal{P}$ is
weakly compact, we have $v(F_n)\downarrow v(F)$ (see Huber and
Strassen \cite{HuSt}). There exists $n_0\in\mathbb{N}$ such that
$v(F_{n_0})\leq v(F)+\varepsilon$.

We have
\begin{align*}
v(X\geq-{1}/{n_0})-1&=v((A_{n_0}\cap O^c)\cup(A_{n_0}\cap O))-1\\
&\leq v(F_{n_0})+v(A_{n_0}\cap O)-1\\
&\leq v(F)+\varepsilon+v(O)-1\\
&\leq v(A)+\varepsilon+v(O)-1\\
&=-(\inf_{P\in\mathcal{P}}P(X<0)-v(O)-\varepsilon)<0
\end{align*}

Since $X\leq 0$ q.s., we have $v(X>t)=0$ for $t\geq 0$. Finally, we
get
\begin{align*}
\E[X]&=\sup_{P\in\mathcal{P}}E_P[X]=\sup_{P\in\mathcal{P}}(\int_0^\infty
P(X>t)\d t+\int_{-\infty}^0
(P(X>t)-1)\d t)\\
&\leq\int_0^\infty\sup_{P\in\mathcal{P}}P(X>t)\d
t+\int_{-\infty}^0(\sup_{P\in\mathcal{P}}P(X>t)-1)\d t
\\&\leq \int_0^\infty v(X>t)\d t+\int_{-\infty}^0
(v(X\geq t)-1)\d t=\int_{-\infty}^0 (v(X\geq t)-1)\d t\\
&\leq \int_{-{1}/{n_0}}^0(v(X\geq t)-1)\d
t\leq(v(X\geq-{1}/{n_0})-1)/n_0<0.
\end{align*}
\end{proof}

In fact, the condition $\inf_{P\in\mathcal{P}}P(X<Y)>0$ is not easy
to verify. But for $X,Y\in Lip(\Omega)$, we can represent them by
$X=\varphi(B_{t_1},B_{t_2}-B_{t_1},\cdots,B_{t_n}-B_{t_{n-1}})$ and
$Y=\psi(B_{t_1},B_{t_2}-B_{t_1},\cdots,B_{t_n}-B_{t_{n-1}})$, where
$\varphi,\psi\in C_{l,Lip}(\mathbb{R}^n)$ and $B$ is the
$G$-Brownian motion on $(\Omega,L_G^1(\Omega),\E)$. When $\x>0$, we
have the following results.

\begin{lemma}
\label{l3} Suppose $\x>0$. Let $X\in Lip(\Omega)$ with the form
$X=\varphi(B_1)$ and $\varphi(x)\leq 0, \forall x\in\mathbb{R}$. If
there exists $x_0\in\mathbb{R}$ such that $\varphi(x_0)<0$, then
$\E[X]<0$.
\end{lemma}

The proof of this lemma depends on some deep estimates of fully
nonlinear parabolic partial different equations, which initially
obtained by Krylov and Safonov \cite{KS}, we will prove it in
Appendix.

\begin{theorem}
\label{t2} Let $\x>0$ and $X,Y\in Lip(\Omega)$ with the forms
$X=\varphi(B_{t_1},B_{t_2}-B_{t_1},\cdots,B_{t_n}-B_{t_{n-1}})$ and
$Y=\psi(B_{t_1},B_{t_2}-B_{t_1},\cdots,B_{t_n}-B_{t_{n-1}})$, where
$\varphi(x)\leq\psi(x), \forall x\in\mathbb{R}^n$. Then
$\E[X]<\E[Y]$ if and only if there exists $x_0\in \mathbb{R}^n$ such
that $\varphi(x_0)<\psi(x_0)$.
\end{theorem}

\begin{proof}
The necessity is obviously, we only need to prove the sufficiency.

We first consider the case of $X=\varphi(B_{t_1},B_{t_2}-B_{t_1})$
and $Y=\psi(B_{t_1},B_{t_2}-B_{t_1})$ with $\varphi\leq\psi$ and
there exists $(x_1,x_2)\in\mathbb{R}^2$ such that
$\varphi(x_1,x_2)<\psi(x_1,x_2)$. By Lemma \ref{l3}, we have
$\E[\varphi(x,B_{t_2}-B_{t_1})]\leq\E[\psi(x,B_{t_2}-B_{t_1})]$ for
each $x\in\mathbb{R}$ and
$\E[\varphi(x_1,B_{t_2}-B_{t_1})]<\E[\psi(x_1,B_{t_2}-B_{t_1})]$.
Let us use Lemma \ref{l3} again, we conclude
$\E[\varphi(B_{t_1},B_{t_2}-B_{t_1})]<\E[\psi(B_{t_1},B_{t_2}-B_{t_1})]$.

For $X=\varphi(B_{t_1},B_{t_2}-B_{t_1},\cdots,B_{t_n}-B_{t_{n-1}})$
and $Y=\psi(B_{t_1},B_{t_2}-B_{t_1},\cdots,B_{t_n}-B_{t_{n-1}})$, we
repeat above procedure, thus $\E[X]<\E[Y]$.
\end{proof}

The condition of $\x>0$ is necessary. When $\s=0$, the above strict
comparison theorem does not hold.
\begin{example}
Let $\x=0$ and $X=B_1\wedge 0$, then we can check that
$u(t,x)=x\wedge 0$ is the unique viscosity solution of $G$-heat
equation: $\partial_t u=\s^2(\partial_{xx} u)^+/2$ with initial
condition $u(0,x)=x\wedge 0$. But we have $\E[X]=u(1,0)=0$. The
strict comparison theorem does not hold.
\end{example}

\begin{corollary}
\label{c1} Let $B_t$ be a $G$-Brownian motion with
$\overline{\sigma}^2=\E[B_1^2]\geq
-\E[-B_1^2]=\underline{\sigma}^2>0$. Then we have
$$\inf_{P\in\mathcal{P}}P(a\leq B_t\leq b)>0,$$
where $-\infty\leq a<b\leq\infty$ and $t>0$.
\end{corollary}
%

\begin{remark}
Suppose $\underline{\sigma}>0$ and $X,Y\in Lip(\Omega)$ satisfy the
same conditions in Theorem \ref{t2}. By Corollary \ref{c1}, we can
prove that $\inf_{P\in\mathcal{P}}P(X<Y)>0$, so by Theorem \ref{t1},
we also have $\E[X]<\E[Y]$.
\end{remark}

In the end, we give another form of strict comparison theorem.

\begin{theorem}
Let $X,Y\in L^1_G(\Omega)$ and  $X\leq Y$ q.s.. If $v(X<Y)>0$ and
$X$ has mean certainty, i.e., $\E[X]=-\E[-X]$, then $\E[X]<\E[Y]$.
\end{theorem}

\begin{proof}
Since $X\leq Y$ q.s. and $v(X<Y)>0$, there exists $P\in\mathcal{P}$
such that $P(X\leq Y)=1$ and $P(X<Y)>0$. Therefore $E_P[X]<E_P[Y].$
Since $\E[X]=-\E[-X]$, we have
$$\E[X]=E_P[X]<E_P[Y]\leq\E[Y].$$
\end{proof}

But unfortunately, in general, if $X\leq Y$ and $v(X<Y)>0$, it does
not imply $\E[X]<\E[Y]$. We give the following counterexample, in
which we will use the notion of quadratic variation process of
$G$-Brownian motion $\langle B\rangle_t$. More details about this
process can be found in Peng
\cite{Peng2006a,Peng2007,Peng2006b,PengSurvey,Peng2009}.
\begin{example}
Let $B_t$ be a $G$-Brownian motion with
$\overline{\sigma}^2=\E[B_1^2]>\underline{\sigma}^2=-\E[-B_1^2]>0$.
we consider $\langle B\rangle_t$ and $\overline{\sigma}^2t$, then we
have $\langle B\rangle_t\leq \overline{\sigma}^2t$ q.s., and we can
choose $P\in\mathcal{P}$ such that $B_t$ becomes the classical
Brownian motion with $E_P[B_1^2]=\underline{\sigma}^2$. Under this
 $P$, the quadratic variation process $\langle B\rangle_t$ equals $\underline{\sigma}^2t$ P-a.s.,
  so we have $P(\langle B\rangle_t<\overline{\sigma}^2t)=1$. It is
easy to get  $v(\langle B\rangle_t<\overline{\sigma}^2t)=1$, but
$\E[\langle
B\rangle_t]=\overline{\sigma}^2t=\E[\overline{\sigma}^2t]$.
\end{example}

\appendix
\section{Proof of Lemma \ref{l3}}

In this section, we complete the proof of Lemma \ref{l3}. We always
suppose $\underline{\sigma}>0$.

The following lemma is a spacial case of Theorem 1.1 in Krylov and
Safonov \cite{KS}. We denote $B_R(x^0):=\{x:|x-x^0|<R\}$,
$B_R:=B_R(0)$ and $Q_{\theta,R}:=(0,\theta R^2)\times B_R$.
\begin{lemma}
\label{l4} Let $\theta>1$ and $R\leq 2$, $u\in
C^{1,2}(Q_{\theta,R})$, $u\geq 0$ be such that
\begin{equation*}
\label{eee}
\partial_tu-a(t,x)\partial_{xx}u=0,\ \ \text{in}\ Q_{\theta,R},
\end{equation*}
where $a\in L^\infty((0,\infty)\times\mathbb{R})$ and for some
$\lambda>0$, $\lambda^{-1}\leq a(t,x)\leq \lambda, \ \forall
(t,x)\in (0,\infty)\times\mathbb{R} $.

Then there is a constant $C$ depending only on $\lambda,\theta$ such
that
$$u(\theta R^2,x)\geq Cu(R^2,0), \ \forall x\in B_{R/2}.$$
\end{lemma}

We now give the proof of Lemma \ref{l3}.

\noindent\textbf{Proof of Lemma \ref{l3}}. We first consider the
case of $\varphi\in C_{b,Lip}(\mathbb{R})$ (bounded and Lipschitzian
continuous). Then the $G$-heat equation (\ref{e1}) has the unique
classical solution, i.e., $u(t,x)\in C^{1,2}((0,2)\times\mathbb{R})$
(see Krylov \cite{Krylov1}). Since $\varphi(x)\leq 0$ for all
$x\in\mathbb{R}$, by the well-known maximal principle, we have
$u(t,x)\leq 0$. Since $\varphi(x_0)<0$, by the continuity of
$u(t,x)$, there exists $\varepsilon>0$ and
$(t_\varepsilon,x_\varepsilon)\in (0,1/2)\times\mathbb{R}$, such
that $u(t_\varepsilon,x_\varepsilon)<-\varepsilon$.

We set $v(t,x)=-u(t,x_\varepsilon-Mx)$, where
$M>2|x_\varepsilon|/\sqrt{t_\varepsilon}$. It is easy to verify that
$v(t,x)$ is the unique solution of the following PDE:
\begin{equation*}
\left\{\begin{array}{cl}
\partial_t v-a(t,x)\partial_{xx}v=0,   &(t,x)\in (0,2)\times\mathbb{R}, \\
v|_{t=0}=-\varphi(x_\varepsilon-Mx),   &x\in\mathbb{R},
\end{array}\right.\
\end{equation*}
where
\begin{equation*}
a(t,x)=\left\{\begin{array}{ll} \overline{\sigma}^2/2M^2&
\text{for}\
(t,x) \ \text{such\ that\ }\ \partial_{xx}u(t,x_\varepsilon-Mx)\geq 0,\\
\underline{\sigma}^2/2M^2&\text{otherwise}.
\end{array}\right.
\end{equation*}

Then $v(t,x)$ satisfies all the condition of above Lemma. We get
$v(1,x_\varepsilon/M)\geq
Cv(t_\varepsilon,0)=-Cu(t_\varepsilon,x_\varepsilon)>0$, where $C>0$
depending on
$\overline{\sigma},\underline{\sigma},x_\varepsilon,t_\varepsilon$.
So we have
$\hat{\mathbb{E}}[\varphi(B_1)]=u(1,0)=-v(1,x_\varepsilon/M)<0$.

For each $\varphi\in C_{l,Lip}(\mathbb{R})$ with $\varphi(x)\leq 0$
and $\varphi(x_0)<0$, we can choose $\varphi'\in
C_{b,Lip}(\mathbb{R})$ such that $\varphi(x)\leq \varphi'(x)\leq 0$
and $\varphi'(x_0)=\varphi(x_0)<0$, then we have
$$\E[\varphi(B_1)]\leq \E[\varphi'(B_1)]<0.$$
$\Box$

\section*{Acknowledgements}
The author would like to  thank Professor  S. Peng for his helpful
discussions.

\end{document}